\documentclass[10pt]{amsart}

\usepackage{enumerate}
\usepackage{amsmath,amssymb}

\newtheorem{theorem}{Theorem}
\newtheorem{lemma}[theorem]{Lemma}

\newtheorem{problem}[theorem]{Problem}

\begin{document}
\title[Spectrally additive maps on  the Wiener algebras]{Spectrally additive maps on the positive cones of  the Wiener algebra}

\author[S.Oi]{Shiho Oi}
\address{Department of Mathematics, Faculty of Science, 
Niigata University, Niigata 950-2181, Japan.}
\email{shiho-oi@math.sc.niigata-u.ac.jp}

\author[K.Sato]{Kaito Sato}
\address{Graduate School of Science and Technology, 
Niigata University, Niigata 950-2181, Japan.}
\email{f24a059h@mail.cc.niigata-u.ac.jp}

\thanks{The first author was supported in part by JSPS KAKENHI Grant Number JP24K06754.
}


\date{}

\begin{abstract}
We study surjective maps between the positive cones of the Wiener algebra that preserve the spectrum of the sum of every two elements. We show that such maps can be extended to isometric real-linear isomorphisms of the Wiener algebra.  
\end{abstract}

\maketitle

\section{Introduction}
Let $\mathbb{T}=\{z \in \mathbb{C} \mid |z|=1\}$. The Wiener algebra  $A(\mathbb{T})$ consists of  all continuous functions on $\mathbb{T}$ whose Fourier series converge absolutely.  That is, 
\[
A(\mathbb{T})=\{ f \in C(\mathbb{T}) \mid \Sigma_{n \in \mathbb{Z}} |\widehat{f}(n)| < \infty \},
\]
where $\widehat{f}(n)=\frac{1}{2\pi}\int_{0}^{2\pi} f(e^{i \theta}) e^{-in \theta} d\theta$ is the $n$th Fourier coefficient of $f$. The norm on $A(\mathbb{T})$ is defined by $\|f\|=\Sigma_{n \in \mathbb{Z}} |\widehat{f}(n)| $ for any $f \in A(\mathbb{T})$. The Wiener algebra is a semi-simple commutative unital Banach algebra and is isometrically algebra isomorphic to the Banach algebra $\ell_1(\mathbb{Z})$ with the isomorphism given by the Fourier transform $f \mapsto \{ \widehat{f}(n)\}_{n \in \mathbb{Z}}$.  

The Fourier algebra $A(G)$ of a (not necessarily abelian) locally compact group $G$, which was introduced by Eymard \cite{E}, is also a semi-simple commutative Banach algebra. The Gelfand spectrum of $A(G)$ is homeomorphic to $G$ (see \cite[Th\'eor\`eme 3.34]{E}). It implies that the spectrum $\sigma(f)$ of any element $f \in A(G)$ coincides with its range. The positive cone $A(G)_{+}$ of the Fourier algebra $A(G)$ is given by $P(G) \cap A(G)$, where $P(G)$ denotes the cone of all continuous positive definite functions on $G$. The Wiener algebra $A(\mathbb{T})$ coincides with the Fourier algebra of the group $\mathbb{T}$. According to Bochner's theorem, we have 
\[
A(\mathbb{T})_{+}=\{ f  \in A(\mathbb{T}) \mid \widehat{f}(n) \ge 0,  \quad \forall n \in \mathbb{Z} \}. 
\]

For a Banach algebra $A$ and $a \in A$, we denote the spectrum of $a$ by $\sigma(a)$. Surjective maps between Banach algebras which preserve spectral properties have been studied extensively in connection with a longstanding open problem called Kaplansky's problem. One of breakthroughs was obtained by Kowalski and S\l odkowski \cite{KS}. They proved that every complex-valued mapping $\phi$ on a complex Banach algebra $A$, without assuming linearity, that satisfies $\phi(a)-\phi(b) \in \sigma(a-b)$ for all $a,b$ in $A$  and $\phi(0)=0$, is a character. Motivated by the theorem,  Havlicek and \v{S}emrl in \cite{HS} investigated bijective maps $T$ on matrix algebras and operator algebras satisfying the condition that $T(a)-T(b)$ is invertible if and only if $a-b$ is invertible. Subsequently, the study of maps on Banach algebras that preserve the spectrum of sums or differences of elements was initiated and further developed (e.g., \cite{ABS, BS, C, LO, MS}).  

In \cite{LO}, Lin and the first author of this paper posed the following question.  
\begin{problem}[{\cite[Problem 1.3]{LO}}]\label{prob}
Let $G$ and $H$ be locally compact groups. 
Let $T: A(G)_{+} \to A(H)_{+}$ be a surjective map such that 
\begin{equation} \label{as}
    \sigma(T(f)+T(g))=\sigma(f+g), \quad f, g \in A(G)_{+}.
\end{equation}
Can $T$ be extended to a positive real-linear isometry from $A(G)$ onto $A(H)$? In addition, are $G$ and $H$ topologically group isomorphic?
\end{problem}

We should first note that a surjective map $T: A(G) \to A(H)$ satisfies
\begin{equation} \label{as0}
    \sigma(T(f)+T(g))=\sigma(f+g), \quad f, g \in A(G),
\end{equation}
if and only if   $T$ is an algebra isomorphism from $A(G)$ onto $A(H)$. 
This statement follows easily from the Kowalski-S\l odkowski theorem, but for completeness we give a short proof here. Let $x \in H$. We define $\phi: A(G) \to \mathbb{C}$ by $\phi(f)=T(f)(x)$. Then $\phi(f)+\phi(-f)=T(f)(x)+T(-f)(x) \in \sigma(T(f)+T(-f))=\sigma(f+(-f))=\{0\}$. Thus $\phi(-f)=-\phi(f)$ for any $f \in A(G)$. We have $\phi(f)-\phi(g)=\phi(f)+\phi(-g)=T(f)(x)+T(-g)(x) \in \sigma(T(f)+T(-g))=\sigma(f-g)$. As $\phi(0)=0$, applying Kowalski-S\l odkowski theorem, we obtain $\phi$ is a character i.e., $\phi$ is in the Gelfand spectrum of $A(G)$. Thus there is $\varphi(x) \in G$ such that $T(f)(x)=\phi(f)=f(\varphi(x))$. The usual routine arguments give that $T$ is an algebra isomorphism. Moreover, the converse is clear. Since every algebra isomorphism is a composition operator $T(f)=f \circ \varphi$, it satisfies \eqref{as0}. Here $\varphi$ is a homeomorphism from $H$ onto $G$, but it is not necessarily a group isomorphism. Thus a surjective map  between Fourier algebras satisfying \eqref{as0} does not necessarily induce a group isomorphism between the underlying groups. 

On the other hand, Problem \ref{prob} is actually true for the Fourier algebra $A(G)$ of any finite group \cite[Theorem 2.5]{LO}. When $G$ and $H$ are finite groups, every surjective map between the positive cones $A(G)_{+}$ and  $A(H)_{+}$ which satisfies \eqref{as} induces a group isomorphism between $G$ and $H$. 

The aim of this paper is to show that Problem \ref{prob} is also true for the Wiener algebra.  More precisely, the main result of this paper is the following.
\begin{theorem}\label{main}
Let $T:A(\mathbb{T})_{+} \to A(\mathbb{T})_{+}$ be a surjective map which satisfies 
\begin{equation}\label{spectrumadditive}
\sigma(Tf+Tg)=\sigma(f+g) \quad  f, g \in A(\mathbb{T})_{+}.
\end{equation}
Then $T(f)=f $ for all $f \in A(\mathbb{T})_{+} $ or $T(f)=\overline{f}$ for all $f \in A(\mathbb{T})_{+} $  holds.
\end{theorem}

\section{Proofs}
Let $G$ be a locally compact group and let $f$ be a positive definite function on $G$ such that  $f(e)=1$, where $e$ is the unit of $G$. Then $f$ is a character of the subgroup $\{x \in G \mid |f(x)|=1\}$ (see for example \cite[Corollary 32.7]{HR}).  Applying this and the fact that the dual group of $\mathbb{T}$ is $\mathbb{Z}$, we obtain the following.
\begin{lemma}\label{1}
Let $\lambda \in \mathbb{R}$ with $\lambda>0$.  Assume that $f \in A(\mathbb{T})_{+}$ with  $\sigma(f) \subset \lambda \mathbb{T}$. Then there is $n \in \mathbb{Z}$ such that $f(z)=\lambda z^{n}$. 
\end{lemma}

For $m,n \in \mathbb{Z}$, we define $M_{m,n}:=\{ z^{m}+z^{n} \mid z \in \mathbb{T} \} \cap 2\mathbb{T}$. 
\begin{lemma}\label{2}
Let $m,n \in \mathbb{Z}$ with $m \neq n$. Then 
\[
M_{m,n}=\{ 2(e^{i\frac{2m\pi}{m-n}})^k \mid 0 \le k \le |m-n|-1\},
\]
and 
\[
\#(M_{m,n})=\frac{|m-n|}{\operatorname{gcd}(m,n)},
\]
where $\#(M_{m,n})$ is the cardinal number of the set $M_{m,n}$ and $\operatorname{gcd}(m,n)$ is the greatest common divisor of $m$ and $n$.
\end{lemma}
\begin{proof}
We assume $m>n$ without loss of generality. For any $z \in \mathbb{T}$, we have $|z^m+z^n|=2 \Longleftrightarrow |z^{m-n}+1|=2 \Longleftrightarrow z^{m-n}=1$. This implies that $z^{m}+z^{n} \in M_{m,n}$ if and only if $z$ is an $(m-n)$-th root of unity. Moreover, in this case we have $z^m=z^n$. Thus, 
\begin{equation*}
\begin{split}
M_{m,n}&=\{ 2(e^{i\frac{2k\pi}{m-n}})^m \mid 0 \le k \le m-n-1 \}\\&=\{ 2(e^{i\frac{2m\pi}{m-n}})^k \mid 0 \le k \le m-n-1\}.
\end{split}
\end{equation*}
Since $\frac{m-n}{\operatorname{gcd}(m-n,m)}$ is coprime with $\frac{m}{\operatorname{gcd}(m-n,m)}$ and $0<\frac{m-n}{\operatorname{gcd}(m-n,m)} \le m-n$, we have $M_{m,n}$ is the set of all $\frac{m-n}{\operatorname{gcd}(m-n,m)}$-th roots of unity. As $\operatorname{gcd}(m-n,m)=\operatorname{gcd}(m,n)$, we get $\#(M_{m,n})=\frac{m-n}{\operatorname{gcd}(m,n)}$.
\end{proof}

We are now in a position to present the proof of Theorem \ref{main}. 
Let $T:A(\mathbb{T})_{+} \to A(\mathbb{T})_{+}$ be a surjective map which satisfies \eqref{spectrumadditive}. 
\begin{lemma}[{\cite[Lemma 2.4]{LO}}]\label{s0}
For any $f \in A(\mathbb{T})_{+}$,  we have $\sigma(Tf)=\sigma(f)$.
\end{lemma}

By Lemma \ref{s0}, for each $n \in \mathbb{Z}$, 
 $\sigma(T(z^n))=\sigma(z^n) \subset \mathbb{T}$. Lemma \ref{1} implies that there is $m \in \mathbb{Z}$ such that $T(z^n)=z^{m}$.  
We define a map $\psi:\mathbb{Z} \to \mathbb{Z}$ by
\[
T(z^{n})=z^{\psi(n)}.
\]

\begin{lemma}\label{3}
For any $\lambda>0$ and $n \in \mathbb{Z}$, we have $T(\lambda z^n)=\lambda z^{\psi(n)}$.
\end{lemma}
\begin{proof}
Let $\lambda>0$ and $n \in \mathbb{Z}$.  
We have $T(\lambda z^n) \in A(\mathbb{T})_{+}$ and $\sigma(T(\lambda z^n))=\sigma(\lambda z^n) \subset \lambda \mathbb{T}$. Together with Lemma \ref{1}, this implies that there exists $m \in \mathbb{Z}$ such that $T(\lambda z^n)=\lambda z^{m}$. Hence $(\lambda+1) \mathbb{T} \supset \sigma(z^n+\lambda z^n)=\sigma(T(z^n)+T(\lambda z^n))=\sigma(z^{\psi(n)}+\lambda z^{m})$. Since $z^{\psi(n)}+\lambda z^{m} \in A(\mathbb{T})_{+}$, Lemma \ref{1} shows that $\psi(n)=m$. 
\end{proof}

\begin{lemma}\label{B1}
The map $\psi:\mathbb{Z} \to \mathbb{Z}$ is a bijection.
\end{lemma}
\begin{proof}
Let $n \in \mathbb{Z}$. Since $T$ is surjective, there exists $f \in A(\mathbb{T})_{+}$ such that $T(f)=z^{n}$. Thus $\sigma(f)=\sigma(T(f))=\sigma(z^n) \subset \mathbb{T}$. Lemma \ref{1} shows that there is $m \in \mathbb{Z}$ such that $f=z^{m}$. Therefore we have $\psi(m)=n$, which implies that $\psi$ is surjective. In order to show that  $\psi$ is injective, suppose that $\psi(m)=\psi(n)$. Then we get $\sigma(z^{m}+z^{n})=\sigma(T(z^m)+T(z^n))=\sigma(z^{\psi(m)}+z^{\psi(n)})=\sigma(2z^{\psi(n)})$. Thus $\sigma(z^{m}+z^{n})$ is either $2\mathbb{T}$ or $\{2\}$. If $\sigma(z^{m}+z^{n})=2\mathbb{T}$, we get $M_{m,n}=2\mathbb{T}$.  By Lemma \ref{2}, we obtain $m=n$.  If $\sigma(z^{m}+z^{n})=\{2\}$, $z^{m}+z^{n}=2$, which yields  $m=n=0$. Hence $\psi$ is injective.
\end{proof}

\begin{lemma}\label{p0}
We have $\psi(0)=0$. 
\end{lemma}
\begin{proof}
Note that $z^0=1$, which is a constant function. Since $\sigma(T(1))=\sigma(1)=\{1\}$, $T(1)=1$.  
\end{proof}

\begin{lemma}\label{psi} 
We have either $\psi(n)=n$ for all $n \in \mathbb{Z}$, or $\psi(n)=-n$ for all $n \in \mathbb{Z}$. 
\end{lemma}
\begin{proof}
Let $m \in \mathbb{Z}\setminus \{1\}$. By \eqref{spectrumadditive}, we have 
\begin{equation}\label{spe1}
\sigma(z+z^{m})=\sigma(z^{\psi(1)}+z^{\psi(m)}).
\end{equation}
Hence we get $M_{1,m}=M_{\psi(1),\psi(m)}$. Lemma \ref{2} shows that $\frac{|m-1|}{\operatorname{gcd}(m,1)}=\frac{|\psi(m)-\psi(1)|}{\operatorname{gcd}(\psi(m),\psi(1))}$. Since $\psi(m) \neq \psi(1)$ by Lemma \ref{B1} and $\operatorname{gcd}(m,1)=1$, there exists $g \in \mathbb{Z}$ such that 
\begin{equation}\label{gcd1}
\psi(m)-\psi(1)=g(m-1).
\end{equation}
Note that $e^{in_1\theta}+e^{in_2\theta}=e^{i\frac{n_1+n_2}{2}\theta}(e^{i\frac{n_1-n_2}{2}\theta}+e^{i\frac{-n_1+n_2}{2}\theta})=2\cos(\frac{n_1-n_2}{2}\theta) e^{i\frac{n_1+n_2}{2}\theta}$. Let $r \in \mathbb{R}$ be an irrational number.  \eqref{spe1} implies that there is $\theta_0 \in [0, 2\pi]$ such that 
\[
2 \cos \left(\frac{m-1}{2} r\pi \right) e^{i\frac{m+1}{2}r \pi}=2 \cos \left(\frac{\psi(m)-\psi(1)}{2} \theta_0 \right) e^{i\frac{\psi(m)+\psi(1)}{2}\theta_0}.
\]
Thus we have
\begin{equation}\label{absolutevalue}
\left|\cos \frac{m-1}{2} r\pi \right|=\left| \cos \frac{\psi(m)-\psi(1)}{2} \theta_0 \right|
\end{equation}
and 
\begin{equation}\label{argument}
e^{i(\frac{\psi(m)+\psi(1)}{2}\theta_0-\frac{m+1}{2}r \pi)} \in \mathbb{R}.
\end{equation}
According to \eqref{absolutevalue},  we have two following cases:\\
{\it{First case} }: $\cos \frac{m-1}{2} r\pi = \cos \frac{\psi(m)-\psi(1)}{2} \theta_0$. Then for some $n \in \mathbb{Z}$,
$\frac{m-1}{2} r\pi =\pm \frac{\psi(m)-\psi(1)}{2} \theta_0+2n\pi$.
By \eqref{gcd1}, we have $\frac{m-1}{2} r\pi =\pm \frac{g(m-1)}{2} \theta_0+2n\pi$, so that $(m-1)(r \pi \mp g\theta_0)=4n\pi$. Since $\mathbb{Z} \ni m \neq 1$, it follows that there is $q \in \mathbb{Q}$ such that 
$g \theta_0=r\pi + q\pi$ or $g \theta_0=-r\pi + q\pi$. \\
{\it{Second case:}} $-\cos \frac{m-1}{2} r\pi = \cos \frac{\psi(m)-\psi(1)}{2} \theta_0$, which is equivalent to $\cos ( \pi - \frac{m-1}{2} r\pi)= \cos \frac{\psi(m)-\psi(1)}{2} \theta_0$. Then there is $n \in \mathbb{Z}$ such that 
$\frac{\psi(m)-\psi(1)}{2} \theta_0=\pm (\pi - \frac{m-1}{2} r\pi)+2n \pi$.
By a similar argument as in the first case, we also obtain $q \in \mathbb{Q}$ such that 
$g \theta_0=r\pi + q\pi$ or $g \theta_0=-r\pi + q\pi$. \\ 
Ultimately, we conclude that \eqref{absolutevalue} implies that there exists $q \in \mathbb{Q}$ such that 
\begin{equation}\label{gt}
g \theta_0=r\pi + q\pi \quad \text{or}  \quad g \theta_0=-r\pi + q\pi.
\end{equation}
 By \eqref{argument}, there is $n \in \mathbb{Z}$ such that 
\[
\frac{\psi(m)+\psi(1)}{2}\theta_0-\frac{m+1}{2}r \pi=n \pi.
\]
Multiplying both sides by $g$ and substituting \eqref{gt}, we obtain 
\[
((m+1)g \mp (\psi(m)+\psi(1)))r =(\psi(m)+\psi(1))q -2ng.
\]
Note that $(m+1)g \mp (\psi(m)+\psi(1))$ is an integer, $r$ is irrational, and the right-hand side is a rational number. Thus we obtain $\psi(m)+\psi(1)=\pm (m+1)g$. When $\psi(m)+\psi(1)=(m+1)g$ (resp.  $\psi(m)+\psi(1)=-(m+1)g$), it follows from \eqref{gcd1} that
$\psi(m)=\psi(1)m$ (resp. $\psi(1)=\psi(m)m$).  Thus it follows that for any $m \in \mathbb{Z} \setminus \{1\} $, 
\[
\psi(m)=\psi(1)m  \quad \text{or} \quad \psi(1)=\psi(m)m
\]
holds. Let $m \in \mathbb{Z}$ with $|m|>|\psi(1)|$. Suppose $\psi(1)=\psi(m)m$, then $\psi(m)=\frac{\psi(1)}{m}$. This contradicts that  $\psi(m) \in \mathbb{Z}$. Thus if  $|m|>|\psi(1)|$, $\psi(m)=\psi(1)m$ holds.  Combining this with Lemma \ref{B1},  we obtain $\psi(1)=1$ or $\psi(1)=-1$. 
Therefore if  $|m|>1$, then 
\begin{equation}\label{eq}
\psi(m)=\psi(1)m.
\end{equation}
By Lemma \ref{p0}, \eqref{eq} holds for $m=0,1$.  Finally, since $\psi$ is a bijection, \eqref{eq}  also holds for $m=-1$. Hence we obtain the desired result.
\end{proof}

\begin{lemma}\label{d1}
Let $f, g \in A(\mathbb{T})_{+}$ satisfy $\sigma(f+ \lambda z^{n})=\sigma(g+\lambda z^{n})$ for any $\lambda >0$ and $n=1,-1$. Then $f=g$.
\end{lemma}
\begin{proof}
Fix $t \in [0, 2\pi)$ with $t \neq 0, \frac{\pi}{2}, \pi, \frac{3}{2}\pi$. Set $I:=f(e^{it})-g(e^{it})$. Since for any positive integer $n$ we have $\sigma(f+ n z)=\sigma(g+n z)$,  there exists $a_n \in [0, 2\pi)$ such that 
\begin{equation}\label{z1}
f(e^{it})+n e^{it}=g(e^{i a_n})+ne^{i a_n}. 
\end{equation}
Since $|e^{i a_n}- e^{it}| \le \frac{1}{n} (\|f\|_{\infty}+\|g\|_{\infty})$, we get $e^{i a_n} \to e^{it}$ as $n \to \infty$.  Set $c_n:=f(e^{it})-g(e^{i a_n})$. The continuity of $g$ implies that $c_n \to I$ as $n \to \infty$. By \eqref{z1}, which is $c_n+n e^{it}=ne^{i a_n}$, we have $n^2=(c_n+n e^{it})(\overline{c_n}+n e^{-it})$. Thus $\operatorname{Re}(c_n e^{-it})=-\frac{1}{2n}|c_n|^2$. As $n \to \infty$, this yields that
\begin{equation}\label{line1}
\operatorname{Re}(I e^{-it})=0. 
\end{equation}
Since for any positive integer $n$ we have $\sigma(f+ n z^{-1})=\sigma(g+n z^{-1})$,  there exists $b_n \in [0, 2\pi)$ such that 
\begin{equation*}\label{z2}
f(e^{it})+n e^{-it}=g(e^{i b_n})+ne^{-i b_n}. 
\end{equation*}
Set  $d_n:=f(e^{it})-g(e^{i b_n})$. Using a similar argument as above, we get $d_n \to I$ as $n \to \infty$ and $\operatorname{Re}(d_n e^{it})=-\frac{1}{2n}|d_n|^2$. Thus we obtain
\begin{equation}\label{line2}
\operatorname{Re}(I e^{it})=0. 
\end{equation}
\eqref{line1} and \eqref{line2} imply that  $I$ is the point of intersection of the two lines.  Hence $I=0$, which means $f(e^{it})=g(e^{it})$. As $f$ and $g$ are continuous, we conclude that $f=g$.
\end{proof}

\begin{proof}[Proof of Theorem \ref{main}]
Let $f \in A(\mathbb{T})_{+}$. Lemma \ref{psi} shows we have two cases. Firstly assume that $\psi(n)=n$ for any $n \in \mathbb{Z}$.  By Lemma \ref{3} we have $\sigma(T(f)+\lambda z^{n})=\sigma(T(f)+ T(\lambda z^{n}))=\sigma(f+\lambda z^{n})$ for any $\lambda>0$ and $n=1,-1$. Applying Lemma \ref{d1}, we conclude that $T(f)=f$. Secondly, assume that $\psi(n)=-n$ for any $n \in \mathbb{Z}$. We define a map $\widetilde{T}: A(\mathbb{T})_{+} \to A(\mathbb{T})_{+}$ by $\widetilde{T}(f)(e^{it})=T(f)(e^{-it})$ for any $e^{it} \in \mathbb{T}$. Then it is easy to check that $\widetilde{T}$ is a surjective map which satisfies \eqref{spectrumadditive}. Moreover we get $\widetilde{T}(z^n)=z^{n}$ for any $n \in \mathbb{Z}$. Using the same argument above, we conclude that $\widetilde{T}(f)=f$. This implies that $T(f)(e^{it})=f(e^{-it})=\overline{f}(e^{it})$ for any $e^{it} \in \mathbb{T}$. 
\end{proof}


\begin{thebibliography}{99}

\bibitem{ABS} M.~Askes, R.~M.~Brits and F.~Schulz, Spectrally additive group homomorphisms on Banach algebras, J. Math. Anal. Appl. {\bf 508} (2022), no.~2, Paper No. 125910, 16 pp.

\bibitem{BS} R.~Benjamin and F.~Schulz, Spectrally additive maps on Banach algebras, Acta Math. Hungar. {\bf 170} (2023), no.~1, 194--208.

\bibitem{C}C.~Costara, Maps on matrices that preserve the spectrum, Linear Algebra Appl. {\bf 435} (2011), no.~11.


\bibitem{E} P.~Eymard, L'alg\`ebre de Fourier d'un groupe localement compact, Bull. Soc. Math. France {\bf 92} (1964), 181--236.

\bibitem{HR}E.~Hewitt and K.~A.~Ross, Abstract harmonic analysis. Vol. II: Structure and analysis for compact groups. Analysis on locally compact Abelian groups. 2nd printing. Berlin: Springer-Verlag (1994).

\bibitem{HS}H.~Havlicek and P.~\v Semrl, From geometry to invertibility preservers, Studia Math. {\bf 174} (2006), no.~1, 99--109.


\bibitem{KS}S.~Kowalski and Z.~S\l odkowski, A characterization of multiplicative linear functionals in Banach algebras, Studia Math. {\bf 67} (1980), no.~3, 215--223.

\bibitem{LO} Y-F.~Lin and S.~Oi, Spectrum additive maps of Fourier algebras, Proc. Amer. Math. Soc. {\bf 153} (2025), no.~8, 3431--3438.

\bibitem{MS} M. ~Mathieu and F. ~Schulz, Additive spectrum preserving mappings from von Neumann algebras, Indag. Math. (N.S.) {\bf 36} (2025), no.~4, 1112--1123.


\end{thebibliography}
\end{document}